\documentclass[reqno,11pt]{amsart}
\usepackage{amsmath,amssymb,amsfonts}

\usepackage{extpfeil}

\usepackage{tikz}
\usetikzlibrary{matrix,arrows}
\usepackage{verbatim}

\usepackage{enumerate}

\usepackage[pagebackref,colorlinks,citecolor={red!50!black},linkcolor={blue!80!black}]{hyperref}

\renewcommand*{\backref}[1]{}
\renewcommand*{\backrefalt}[4]{%
    \scriptsize%
    {
    \ifcase #1 (\textcolor{red}{Uncited.})%
          \or (Cited\ on p.~#2).%
          \else (Cited\ on pp.~#2).%
    \fi%
    }
}

\usepackage{orcidlink}

\usepackage{fourier}
\usepackage[varqu,varl]{inconsolata}
\usepackage[utf8]{inputenc}

\title{Localizing and colocalizing subcategories on schemes}

\author[L. Alonso]{Leovigildo Alonso Tarr\'{\i}o \orcidlink{0000-0002-6896-0652}}
\address[L. A. T.]{CITMAGA\\
Departamento de Matem\'a\-ticas\\
Universidade de Santiago de Compostela\\
E-15782  Santiago de Compostela, Spain}
\email{leo.alonso@usc.es}

\author[A. Jerem\'{\i}as]{Ana Jerem\'{\i}as L\'opez \orcidlink{0000-0001-7964-1334}}
\address[A. J. L.]{CITMAGA\\
Departamento de Matem\'a\-ticas\\
Universidade de Santiago de Compostela\\
E-15782  Santiago de Compostela, Spain}
\email{ana.jeremias@usc.es}

\author[E. Loureiro]{Eduardo Loureiro Novo \orcidlink{0009-0009-6001-1928}}
\address[E. L. N.]{CITMAGA\\
Departamento de Matem\'a\-ticas\\
Universidade de Santiago de Compostela\\
E-15782  Santiago de Compostela, Spain}
\email{loureironovoe@gmail.com}

\thanks{This work has been partially supported by Xunta de Galicia's ED431C 2023/31 project with E.U.'s FEDER funds.}


\date{June 5, 2026,
 \emph{typeset}, \today}

\theoremstyle{plain}
\newtheorem{theorem}{Theorem}[section]
\newtheorem{proposition}[theorem]{Proposition}
\newtheorem{lema}[theorem]{Lemma}
\newtheorem{lemma}[theorem]{Lemma}
\newtheorem{corollary}[theorem]{Corollary}

\theoremstyle{remark}
\newtheorem{remark}[theorem]{Remark}
\newtheorem*{remark*}{Remark}

\theoremstyle{definition}
\newtheorem{definition}[theorem]{Definition}
\newtheorem*{ack}{Acknowledgements}
\newtheorem{cosa}[theorem]{}			

\numberwithin{equation}{theorem}

\newcommand{\CA}{\mathcal{A}}

\newcommand{\CC}{{\mathcal C}}

\newcommand{\CE}{{\mathcal E}}
\newcommand{\CF}{{\mathcal F}}
\newcommand{\CG}{{\mathcal G}}
\newcommand{\CH}{{\mathcal H}}

\newcommand{\CJ}{{\mathcal J}}
\newcommand{\CK}{{\mathcal K}}

\newcommand{\CM}{\mathcal{M}}
\newcommand{\CN}{\mathcal{N}}
\newcommand{\CO}{\mathcal{O}}
\newcommand{\CP}{\mathcal{P}}
\newcommand{\CQ}{\mathcal{Q}}

\newcommand{\CT}{\mathcal{T}}
\newcommand{\CV}{\mathcal{V}}

\newcommand{\CW}{\mathcal{W}}

\newcommand{\SA}{\mathsf{A}}

\newcommand{\SC}{\mathsf{C}}
\newcommand{\Col}{\mathsf{Col}}
\newcommand{\SD}{\mathsf{D}}

\newcommand{\SL}{\mathsf{L}}
\newcommand{\Loc}{\mathsf{Loc}}
\newcommand{\Locset}[1]{\mathsf{Loc}({#1})}
\newcommand{\Locsett}[1]{\mathsf{Loc}_{\otimes}({#1})}
\newcommand{\Colset}[1]{\mathsf{Col}({#1})}

\renewcommand{\SS}{\mathsf{S}}

\newcommand{\ST}{\boldsymbol{\mathsf{T}}}

\newcommand{\SY}{\mathsf{Y}}
\newcommand{\SZ}{\mathsf{Z}}
\newcommand{\Ab}{\mathsf{Ab}}

\newcommand{\D}{\boldsymbol{\mathsf{D}}}
\newcommand{\K}{\boldsymbol{\mathsf{K}}}
\newcommand{\LL}{\boldsymbol{\mathsf{L}}}
\newcommand{\R}{\boldsymbol{\mathsf{R}}}
\newcommand{\A}{\mathsf{A}}

\newcommand{\NN}{\mathbb{N}}
\newcommand{\ZZ}{\mathbb{Z}}

\newcommand{\ia}{{\mathfrak a}}

\newcommand{\ip}{{\mathfrak p}}

\newcommand{\im}{{\mathfrak m}}

 \newcommand{\hocolim}[1]{\begin{array}[t]{c} {\rm hocolim}\\[-7.5 pt]
 {\lto} \\[-7.5 pt] {\scriptstyle {#1}} \end{array}}
\newcommand{\dirlim}[1]{\begin{array}[t]{c} {\rm lim}\\[-7.5 pt]
 {\longrightarrow} \\[-7.5 pt] {\scriptstyle {#1}} \end{array}}

\newcommand{\lto}{\longrightarrow}

\newcommand{\inc}{\hookrightarrow}

\newcommand{\limp}{\Longrightarrow}
\newcommand{\ldimp}{\Longleftrightarrow}

\DeclareMathOperator{\Hom}{Hom}
\DeclareMathOperator{\shom}{\CH\mathit{om}}
\DeclareMathOperator{\rshom}{\R\!\shom}
\DeclareMathOperator{\dhom}{\CH\mathsf{om}}

\DeclareMathOperator{\rhom}{\R{}Hom}

\DeclareMathOperator{\Img}{Im}

\DeclareMathOperator{\spec}{Spec}

\DeclareMathOperator{\supp}{supp}

\DeclareMathOperator{\h}{H}
\DeclareMathOperator{\id}{id}

\newcommand{\qc}{\mathsf{qc}}
\newcommand{\op}{\mathsf{op}}

\newcommand{\HOMqc}{\R\dhom^{\qc}}
\newcommand{\productoqc}{\prod^{{}^{\qc}}}
\newcommand{\coherador}{\gamma^{\qc}}

\DeclareMathOperator{\Spec}{Spec}

\newcommand{\ie}{{\it i.e.~}}

\newcommand{\lc}{{\it loc.cit.\/} }

\begin{document}

\begin{abstract}  A full triangulated subcategory $\SL \subset \ST$ of triangulated category $\ST$ is \emph{localizing} if it is stable for coproducts. If, further,  $\ST$ is $\otimes$-triangulated, we say that $\SL$ is $\otimes$-ideal if $F \otimes G \in \SL$ for all $G \in \SL$ and all $F \in \ST$. Analogously, a full triangulated subcategory $\SC \subset \ST$ is \emph{colocalizing} if it is stable for products. If, further,  $\ST$ is \emph{closed}, \ie $\otimes$-triangulated with internal homs (denoted $[-,-]$), we say that $\SC$ is $\CH$-coideal if $[F, G] \in \SC$ for all $G \in \SC$ and all $F \in \ST$.

For a point generated concentrated scheme $X$, we prove that all $\otimes$-ideal localizing subcategories of $\D_{\qc}(X)$ are classified by the subsets of $X$. As a consequence, we prove that for $\CH$-coideal colocalizing subcategories of $\D_{\qc}(X)$ the same holds. Moreover, every such colocalizing subcategory $\SC$ is of the form $\SC = \SL^\perp$, where $\SL$ is a $\otimes$-ideal localizing subcategory of $\D_{\qc}(X)$. 
\end{abstract}

\maketitle
\tableofcontents

\section*{Introduction}

In the seminal paper \cite{Nct}, Neeman proved that the localizing subcategories of the derived category $\D(R)$ of a commutative noetherian ring $R$ are classified by the subsets of $\Spec(R)$. This classification was inspired by the chromatic tower in stable homotopy and previous ideas by Bousfield, Hopkins and Ravenel. That paper inaugurated a flurry of activity in generalizing these results to several contexts, one of the most outstanding is the geometric case of schemes. Later Neeman, in \cite[Corollary 2.8]{Ncs}, classified colocalizing subcategories in $\D(R)$. There has been several generalizations of Neeman's results. 

One key observation is that both the classification theorems and its generalizations rely strongly on the noetherian hypothesis. As a first attempt to classify colocalizing subcategories on schemes, we tried to generalize Neeman's strategy using global Koszul complexes (which exist on divisorial schemes), this strategy depends again on the noetherian hypothesis. In the present paper, we propose a different approach to these results taking advantage of the existence of a clear notion of point in the context of schemes. As long as the residue fields determine the derived category, we will see that the analogues of Neeman's theorems hold and the proofs are actually pretty natural.

Before delving deeper in our methods and results, let us comment the previous work known in this realm we are aware of. 

Let us start now with some details on the relevant definitions. A localizing subcategory of a triangulated category $\ST$ is a full triangulated subcategory of $\ST$ stable for coproducts. These categories are important because they arise as kernels of localization functors. When $\ST = \D(\SA)$ is the derived category of a Grothendieck category $\SA$, a localizing subcategory generated by a set of objects in $\D(\SA)$ is the kernel of a localization functor (see \cite[Theorem 5.7]{AJS1}). A \emph{colocalizing subcategory} is, dually, a full triangulated subcategory stable for \emph{products}. It is not known if there is a condition on the generation of a colocalizing subcategory to be the kernel of a colocalization functor.

The initial boost on this topic was \cite{Nct}. In this paper, Neeman proved that the localizing subcategories of the derived category $\D(R)$ of a commutative noetherian ring $R$ are classified by the subsets of $\Spec(R)$. In 2004, Alonso, Jeremías and Souto in \cite{AJS3} generalized this result to the context of noetherian schemes, showing that the localizing subcategories in $\D_{\qc}(X)$ (the derived category of complexes of sheaves on $X$ with quasi-coherent homology) that are $\otimes$-ideal, are in bijection with the subsets of $X$. Later, a different proof was given by Stevenson in 2013 \cite[Corollary 8.13.]{St13}.

If $\SL \subset \ST$ is a localizing subcategory of a triangulated category $\ST$ then $\SL^\perp$ is obviously colocalizing. Therefore, when $\ST = \D(R)$, Neeman's result for $\D(R)$ already gives a bunch of colocalizing subcategories of $\D(R)$.  A natural question arose: Are there more colocalizing subcategories in $\D(R)$? The answer was provided by Neeman in 2011, \cite[Corollary 2.8]{Ncs}: There are no further colocalizing subcategories in $\D(R)$.

After these developments it has arisen the interest on classification of localizing and colocalizing subcategories in the context of tensor-triangulated geometry. The starting point was the paper by Benson, Iyengar and Krause \cite{bik}. Their results require the action of the derived category of a ring in the category under consideration, something that is not always possible in the case of schemes, even in the noetherian case.

More recently, Verasdanis in \cite[Theorem 6.9]{v}, addresses the question for a noetherian separated scheme with very different techniques from those used here. A preprint by Barthel, Castellana, Heard and Sanders deals with stratified and costratified tt-categories. One of its applications is \cite[Theorem 19.5]{bchs}. It provides a classification of certain colocalizing subcategories of $\D_{\qc}(X)$ when $X$ is a \emph{topologically weakly noetherian }concentrated scheme \emph{with} $\D_{\qc}(X)$ \emph{stratified}. The methods in \lc are those of tensor-triangulated geometry, thus a substitute for sheaves of residue fields is provided through the use of certain idempotent functors, determined by localization and acyclization functor associated to compact objects, that play the role ---but are quite different--- from true residue fields.

Now it is time to describe our classification results. We work in the context of concentrated (\ie quasi-compact and quasi-separated) schemes. Our  basic hypothesis is that the scheme is \emph{point generated}, \ie the category $\D_{\qc}(X)$ is generated by the sheaves of residue fields of its points. This notion is discussed in detail in \S \ref{sec3}. noetherian schemes are point generated (Theorem \ref{LocK(x)generan}) thus this hypothesis generalizes the previous results. There are examples by Neeman of non-noetherian schemes that are not point generated \cite{Oddball}. However, there are plenty of non-noetherian point generated zero dimensional affine schemes as the example by Stevenson in \cite{St14} shows.

Our classification of localizing subcategories (generalizing \cite[Theorem 4.12]{AJS3} for the case of usual schemes), is stated as follows:

\begin{theorem}
(\emph{Theorem \ref{ThclasB}})

Let $X$ be a point generated concentrated scheme. 
There is a bijection
\[
\left\{    
    \begin{gathered}    
        \text{ Localizing }\\
        \otimes\text{-ideals of\/ } \D_{\qc}(X)
    \end{gathered}
\right\}
\xtofrom[\Loc]{\SZ}
\left\{  
    \begin{gathered}    
        \text{ Subsets of }\\
        X
    \end{gathered}
\right\}
\]
where the map $\Loc$ takes a subset $Z \subset X$ to the localizing subcategory $\Locset{Z}$ generated by the set of sheaves $\{\CK(x) \,\,|\,\, x \in Z\}$ and $\SZ$ takes a $\otimes$-ideal localizing subcategory
$\SL$ to the set of points $\SZ(\SL) := \{\,x\in X \,\,|\,\, \CK(x)\in \SL\,\}$.

\end{theorem}

The property to be a $\otimes$-ideal characterizes thus those localizations deter\-mi\-ned by subsets of underlying set of a scheme. In the context of affine schemes, \ie of derived categories of rings, this issue does not arise because all localizing subcategories are $\otimes$-ideals, see Lemma \ref{allideals}.

In the case of schemes a subcategory right orthogonal for the $\rhom$ pairing to a $\otimes$-ideal localizing subcategory is an $\CH$-coideal colocalizing subcategory. As in the case of localizing subcategories all colocalizing subcategories of the derived category of a ring are $\CH$-coideals by Lemma \ref{allcoideals}. These are classified in the second main Theorem of the present paper. Specifically,

\begin{theorem}
(See \emph{Theorem \ref{ThclasCB}}).

Let $X$ be a point generated concentrated scheme. 
There is a bijection
\[
\left\{    
    \begin{gathered}    
        \text{ Colocalizing }\\
        \CH\text{-coideals of\/ } \D_{\qc}(X)
    \end{gathered}
   \right\}
\xtofrom[\Col]{\SY}
\left\{  
\begin{gathered}    
        \text{ Subsets of }\\
        X
    \end{gathered}
   \right\}
\]
where the map $\Col$ takes a subset $Y \subset X$ into the colocalizing subcategory $\Colset{Y}$ generated by the set of sheaves $\{\,\CK(x) \,\,|\,\, x \in Y\,\}$ and $\SY$ assigns to an $\CH$-coideal colocalizing subcategory $\SC$, the set  $\SY(\SC) := \{\,x\in X \,\,|\,\, \CK(x)\in \SC\,\}$.
\end{theorem}

An interpretation of Theorems 0.1 and 0.2, is that taking orthogonals to localizing  and colocalizing subcategories of $\D_\qc(X)$ is expressed to taking complements in $X$ in the following sense, for $Z \subset X$ and $Y = X \setminus Z$ we have, with the previous notations, that
\[
\Locset{Z}^\perp = \Colset{Y} \qquad \Locset{Z}={}^\perp\Colset{Y}.
\] 

Our assumption that the schemes are point generated not only gives a broader class of cases covered but also allowed us to give simpler proofs, already for the previously considered classification of $\otimes$-ideal localizing subcategories. Also, the present proof of the classification of localizing and colocalizing subcategories are simpler and shorter than those in \cite{Nct} and \cite{Ncs} for a noetherian affine scheme $X = \Spec(R)$.

Let us discuss briefly the organization of the paper. In the first section we establish the basic definitions on schemes and their derived categories of complexes of sheaves with quasi-coherent homology. In section \ref{sec2} we give the definitions of localizing and colocalizing subcategories and  establish the notation for the sheaves associated to residue fields of points. We study the basic behavior of residue fields under tensor product and the complex of homomorphisms. We also remark that such localizing categories are associated to a localization functor. As mentioned before, in section \ref{sec3} we introduce the notion of point generated scheme and check that noetherian schemes are point generated.

The remaining sections are devoted to the main results. In section \ref{sec4} we give a direct proof of the characterization of $\otimes$-ideal localizing subcategories. In section \ref{sec5} we prove the result that extends the main one in \cite{Ncs}. We classify the colocalizing subcategories of $\D_{\qc}(X)$ that are $\CH$-coideal by putting them in bijection with subsets of $X$. The complementary subset gives a $\otimes$-ideal localizing subcategory whose acyclization corresponds to the former colocalizing subcategory. 

The last section \ref{sec6} is an addendum putting in perspective our condition of point generation. We show that it sits strictly in between the noetherian condition and the support property, \ie that complexes with empty (small) support are zero.

\begin{ack}
 We thank comments and pointers to the recent literature given to us by Rudradip Biswas, Natalia Castellana and Amnon Neeman.
\end{ack}

\section{Geometric preliminaries}

We recall some basic results about quasi-coherent sheaves on general concentrated schemes. Recall that a scheme $X$ is \emph{concentrated} if it quasi-compact scheme and quasi-separated (\ie the intersection of two quasi-compact open subsets in $X$ is quasi-compact). Following \cite[\S6.1]{GD}, we say that $f\colon X\to Y$ is a quasi-compact morphism  if for every quasi-compact open subset $U \subset X$, $f^{-1}(U)$ is quasi-compact. Also we say that $f$ is quasi-separated if its diagonal embedding $\delta_f \colon X \to X \times_Y X$ is quasi-compact. We call a morphism $f$ \emph{concentrated} if it is quasi-compact and quasi-separated. Notice that if $Y$ is concentrated then a morphism $f\colon X\to Y$  is concentrated if and only if $X$ is concentrated.

For a scheme $X$, $\SA(X)$ will denote the category of sheaves of $\CO_X$-Modules and $\A_\qc(X)$ the full subcategory of quasi-coherent shaves of $\CO_X$-Modules. The category of quasi-coherent sheaves defined on quasi-compact, quasi-separated schemes is well-behaved. Both categories $\A(X)$ and $\A_\qc(X)$ are Grothendieck categories. Their corresponding derived categories will be denoted by $\D(X)=\D(\A(X))$ and $\D(\A_\qc(X))$, respectively. We also denote by $\D_\qc(X)$ the full subcategory of $\D(X)$ whose objects are the complexes of sheaves of $\CO_X$-Modules with quasi-coherent homology.

\begin{cosa}\label{cohfunc}
\textbf{Complexes of quasi-coherent sheaves and the coherator functor}.
In our setting, the inclusion $\mathsf{i} \colon \A_\qc(X) \to \A(X)$ has a right adjoint $Q_X \colon \A(X) \to \A_\qc(X)$ called the \emph{coherator} functor (see \cite[p. 187, Lemme 3.2]{erg} or, for another treatment, \cite[Appendix B]{tt}). The inclusion $\mathsf{i} \colon \A_\qc(X) \to \A(X)$ is an exact functor so it extends to a $\Delta$-functor ${\mathsf{i}} \colon \D(\A_\qc(X)) \to \D(X)$. By \cite[p. 187, Theorem 5.4]{AJS1}, in the category $\K(X)$ there are $K$-injective resolutions, so it is possible to define the right derived functor $\R{}Q_X \colon \D(X) \to \D(\A_\qc(X))$, the right adjoint of the inclusion ${\mathbf{\mathsf{i}}} \colon \D(\A_\qc(X)) \to \D(X)$.  This functor 
factors through $\D_\qc(X)$. 
If the concentrated scheme $X$ is furthermore semi-separated (that is, the intersection of two affine opens of $X$ is still affine), $\mathsf{i} \colon \D(\A_\qc(X)) \to \D_\qc(X)$ is an equivalence of categories with quasi-inverse  $\R{}Q_X$ (\emph{cf.} \cite[Corollary 5.5]{BN} or \cite[Proposition (1.3)]{AJL}). In particular for an affine scheme $X = \Spec(R)$, we have that $\D_\qc(X) \cong \D(R)$. Along the paper we will identify both categories without further mention.
\end{cosa}

\begin{cosa} \label{prodqc}
\textbf{Complexes of sheaves with quasi-coherent homology}.
Let us asume that $X$ is a general concentrated scheme.  The full subcategory $\D_\qc(X)\subset \D(X)$ is a triangulated subcategory closed for coproducts, that is,  it is a localizing subcategory of $\D(X)$. Furthermore,  $\D_\qc(X)$ is compactly generated \cite[Theorem 3.1.1]{BB}{\footnote{The compact objects in $\D_\qc(X)$ are precisely the perfect complexes (see, for instance, \cite[Proposition 4.7]{asht}). Let us recall that a complex $\CP\in \D(X)$ is  perfect if it is locally quasi-isomorphic to a bounded complex of vector bundles. The result in \cite[Theorem 3.1.1]{BB} is more precise establishing that $\D_\qc(X)$ is generated by a single perfect complex.
}}.
By \cite[Theorem 5.7]{AJS1} (using \cite[Propositions~1.6 and 1.7]{AJS1}), the inclusion $\D_\qc(X)\subset \D(X)$ possesses a right adjoint, that is, there exists an \emph{acyclization functor} $\coherador\colon  \D(X)\to \D(X)$ such that for $\CF\in \D(X)$ it holds that $\coherador(\CF)=\CF$ if and only if $\CF\in \D_\qc(X)$.
In particular, the category $\D_\qc(X)$ has all products. Given $\{\CF_i\,\,|\,\, i\in I\}$ a family of objects in $\D_\qc(X)$, $\coherador \prod_{i\in I}\CF_i$ is the product in $\D_\qc(X)$, that we will denote by $\productoqc_{i\in I}\CF_i$. Note that the inclusion functor $\iota\colon \D_\qc(X)\to \D(X)$ preserves all coproducts but in general it does \emph{not} preserve all products. 
\end{cosa}

\begin{cosa}\label{monstr}
\textbf{Monoidal closed structure in $\D(X)$}.
For any scheme $X$, the category $\D(X)$ is endowed with a structure of \emph{symmetric monoidal category}, the left derived tensor product,
 \[
 -\otimes^{\LL}_{X}-\colon \D(X)\times \D(X)\lto \D(X),
\] 
is the \emph{tensor product} structure, and $\CO_X$ is the unit object.  Furthermore, $\D(X)$ is a \emph{closed category} with the \emph{internal hom} structure defined by the right derived sheaf-hom functor
\[
\rshom_X (-, -) \colon \D(X)^{\op}\times \D(X)\lto \D(X).
\] 
The natural adjunction isomorphism
\[
\Hom_{\D(X)} (\CA\otimes^{\LL}_{X} \CF,\CG) \cong \Hom_{\D(X)}(\CA, \rshom_X (\CF,\CG)) \quad (\CA ,\, \CF,\, \CG \in\D(X)),
\]
induces an $\Delta$-functorial isomorphism
\[
\R\shom_{X} (\CA\otimes^{\LL}_{X} \CF,\CG) \cong \R\shom_{X}(\CA, \rshom_X (\CF,\CG)) \quad (\CA ,\, \CF,\, \CG \in\D(X)).
\]
For a detailed discussion on these structures and its consequences see \cite[Chapters 2 and  3 (\S 3.4)]{yellow}.
\end{cosa}

\begin{cosa}\label{ayuda1}
\textbf{Monoidal closed structure in $\D_\qc(X)$}.
By \cite[Proposition~3.9.1]{yellow}, if $\CF,\,\CG\in \D_\qc(X)$ then   $\CF\otimes^{\LL}_{X} \CG\in \D_\qc(X)$. So the structure of symmetric monoidal category of $\D(X)$ defines on  $\D_\qc(X)$ a structure of \emph{symmetric monoidal category}.

The internal hom in $\D_\qc(X)$ is the $\Delta$-bifunctor
\[
\HOMqc_X(\,-\,,\,-\,)\colon \D_\qc(X)^{\op} \times  \D_\qc(X) \lto  \D_\qc(X),
\]
defined, for $\CF,\,\CG\in \D_\qc(X)$, by
\[
\HOMqc_X(\CF,\CG) := \coherador  \rshom_X (\CF, \CG).
\]
Indeed, if $X$ is a concentrated scheme, the natural adjunction isomorphism
\[
\Hom_{\D(X)} (\CA\otimes^{\LL}_{X} \CF,\CG) \cong \Hom_{\D(X)}(\CA, \rshom_X (\CF,\CG))
\]
induces an isomorphism
\[
\Hom_{\D(X)}(\CA\otimes^{\LL}_{X}\CF,\CG) \cong \Hom_{\D(X)}(\CA, \coherador \rshom_X (\CF,\CG)).
\]
for all $\CA \in\D_\qc(X)$ and any $ \CF,\, \CG \in\D(X)$. As a consequence we obtain the desired $\Delta$-functorial isomorphism
\[
\Hom_{\D(X)} (\CA\otimes^{\LL}_{X} \CF,\CG) \cong \Hom_{\D(X)}(\CA, \HOMqc_X (\CF,\CG)) \quad (\CA ,\, \CF,\, \CG \in\D_\qc(X)).
\]
As before, it induces an $\Delta$-functorial isomorphism
\[
\HOMqc_X (\CA\otimes^{\LL}_{X} \CF,\CG) \cong \HOMqc_X(\CA, \HOMqc_X (\CF,\CG)) \quad (\CA ,\, \CF,\, \CG \in \D_\qc(X)).
\]

The following lemma is an immediate consequence of the existence of the adjunction isomorphisms.
\end{cosa}

\begin{lemma} For a concentrated scheme $X$ and for any $\CF \in \D_\qc(X)$, the functor $\HOMqc_X(\CF,-) \colon \D_\qc(X)\to \D_\qc(X)$ preserves products, and the contravariant functor $\HOMqc_X(-,\CF) \colon \D_\qc(X)\to \D_\qc(X)$ transforms coproducts into products.
\end{lemma}

\begin{cosa}\label{inverse-direct}
{\textbf{Direct and inverse image functors}.}
For any scheme map $f\colon X\to Y$, it holds that $\LL f^*(\D_\qc (Y))\subset \D_\qc(X)$ \cite[Proposition~3.9.1]{yellow}.
If $f$ is concentrated then $\R f_*(\D_\qc (X))\subset \D_\qc(Y)$ \cite[Proposition~3.9.2]{yellow}, furthermore the bifunctorial projection map
\[
p\colon \CA\otimes^{\LL}_{\CO_X}\R f_*\CF\lto \R f_*(\LL f^*\CA\otimes^{\LL}_{\CO_X}\CF)
\]
is an isomorphism for all $\CA\in \D_\qc(Y)$ and $\CF\in \D_\qc(X)$ \cite[Proposition~3.9.4]{yellow}.

Let $f\colon X\to Y$ be a concentrated morphism of schemes, and let us assume that $X$ is concentrated. Let $\R f_*\colon \D_\qc(X)\to \D(Y)$ be the restriction of the right derived functor $\R f_*\colon \D(X)\to \D(Y)$.  The functor $\R f_*\colon \D_\qc(X)\to \D(Y)$ possesses a right $\Delta$-adjoint $f^\times\colon  \D(Y) \to \D_\qc(X)$, the twisted inverse image functor (\emph{cf.} \cite[$\S 4$]{yellow}\footnote{Alternatively, the category $\D_\qc(X)$ is compactly generated (\cite[Theorem 3.1.1]{BB}), the functor $\R f_*\colon \D_\qc(X)\to \D(Y)$ preserves coproducts \cite[Corollary (3.9.3.3)]{yellow} so, by Brown representability theorem (as in \cite[Theorem 3.1]{Ngd}) it has a right adjoint $f^\times\colon \D(Y)\to \D_\qc(X)$.}.
The adjunction isomorphism together with the projection isomorphism induces isomorphisms
\[
\Hom_{\D(Y)}\big(\CA, \R\shom_Y\big(\R f_*\CF, \CG \big)\big) \cong \Hom_{\D(Y)}\big(\CA,  \R f_* \R\shom_X\big(\CF, {f}^\times \CG\big) \big),
\]
for all $\CA,\,\CF\in \D_\qc(X)$ and $\CG\in \D(X)$ (\cite[Theorem~4.1.1]{yellow}).
So the adjunction between the restricted functors  $\R f_*\colon \D_\qc(X)\to \D_\qc(Y)$ and  $f^\times\colon \D_\qc(Y) \to \D_\qc(X)$  may be upgraded to an \emph{internal adjunction} between closed categories, that is, there is an unique $\Delta$-bifunctorial isomorphism
\[
\HOMqc_{Y}\big(\R f_*\CF, \CG \big) \cong  \R f_* \HOMqc_X\big(\CF, {f}^\times \CG \big),
\]
for $\CF\,\in \D_\qc(X)$ and $\CG\in \D_\qc(Y).$
\end{cosa}

\section{Localizing and colocalizing subcategories and points.}\label{sec2}

We remind the reader that $X$ stands for a concentrated scheme. 

A \emph{localizing} subcategory in $\D_{\qc}(X)$ is a full triangulated subcategory of $\D_{\qc}(X)$ closed under coproducts. Dually, a \emph{colocalizing} subcategory in $\D_{\qc}(X)$ is a full triangulated subcategory  of $\D_{\qc}(X)$ closed under products. 

By the well-known \emph{Eilenberg's swindle}, both localizing and colocalizing subcategories  are stable for direct summands. Indeed, let $\CF = \CF_1 \times \CF_2$ in $\D_{\qc}(X)$, using the ismorphism $\CF_1 \times \CF_2 \cong \CF_2 \times \CF_1$, there is a distinguished triangle $\D_{\qc}(X)$ 
 \[
 \CF_1 \lto \prod_{\mathbb N}{}^{\qc} \,\CF 
 \lto \prod_{\mathbb N}{}^{\qc} \,\CF \overset{+}\lto
 \]
where the second morphism corresponds to the canonical
projection 
\[
\CF_1 \times (\CF_2 \times \CF_1 \times \CF_2 \times \cdots) \lto \CF_2 \times \CF_1
\times \CF_2 \times \CF_1 \times \cdots\] 
whence $\CF_1$ belongs to the colocalizing subcategory that contains $\CF$. The argument is similar for localizing subcategories.
 
\begin{cosa}
\label{ortcoloc}
For  a class of objects $\SS \subset \D_{\qc}(X)$, let us denote by $\Loc(\,\SS\,)$ the smallest localizing subcategory of $\D_{\qc}(X)$ that contains $\SS$, and by $\Col(\,\SS\,)$ the smallest colocalizing subcategory containing $\SS$. We will say that $\Loc(\,\SS\,)$ is the \emph{localizing subcategory} of $\D_{\qc}(X)$ \emph{generated by}  $\SS$, and $\Col(\,\SS\,)$ is the \emph{colocalizing subcategory} of $\D_{\qc}(X)$ \emph{generated by} $\SS$.\footnote{Note that $\Loc(\,\SS\,)$ is also the smallest localizing subcategory of $\D(X)$ that contains $\SS$, because $\D_{\qc}(X)$ is a localizing subcategory of $\D(X)$ (see \ref{prodqc}) and $\SS \subset \D_{\qc}(X)$. But in general $\Col(\,\SS\,)$ is \emph{not} a colocalizing subcategory of $\D(X)$.}  These definitions agree in the affine case with Neeman's definition in \cite[(i) beginning \S 1]{Ncs}.

For any class of objects $\SS\subset\D_{\qc}(X)$, let us consider the corresponding orthogonal subcategories with respect to the $\Delta$-bifunctor
\[
 \R\Hom_{X}(-,-)\colon \D_{\qc}(X)\times \D_{\qc}(X) \rightarrow  \D(\Ab).
 \]
Let $\SS^\perp$ be the right ortogonal of $\SS$, and  ${}^{\perp}\SS$ be its left ortogonal in $\D_{\qc}(X)$:
\begin{align*}
\SS^\perp     & := \left\{ \CF\in\D_{\qc}(X) \,\,|\,\, \R\Hom_{X}(\CG,\CF)=0,\ \forall \CG \in \SS \right\},\\
{^{\perp}\SS} & := \left\{ \CF\in\D_{\qc}(X) \,\,|\,\, \R\Hom_{X}(\CF,\CG)=0,\ \forall \CG \in \SS \right\}.
\end{align*}

\noindent The right ortogonal $\SS^\perp$ is a colocalizing subcategory of $\D_{\qc}(X)$, and the left ortogonal ${}^{\perp}\SS$ is a localizing subcategory. Note that
\[
\Loc(\,{\SS}\,) \subset {}^\perp \big(\SS^\perp\big),\qquad\Col(\,{\SS}\,) \subset \big({}^\perp \SS\big)^\perp.
\]
If the embedding $\Loc(\,{\SS}\,)\subset \D_{\qc}(X)$ has a right adjoint then $\Loc(\,{\SS}\,)={}^\perp \big(\SS^\perp\big)$. Similarly, if the embedding ${}^\perp \SS\subset \D_{\qc}(X)$ has a left adjoint then $\Col(\,{\SS}\,) = \big({}^\perp \SS\big)^\perp$, this is discussed in \cite[Proposition 1.6]{AJS1}. For a criterion for the existence of adjoints, see \cite[Theorem 5.7]{AJS1}.
\end{cosa}

\begin{cosa}
Let us start by analyzing the simplest case, the affine scheme $X = \Spec(k)$ with $k$ a field. The category $\D(\spec(k))=\D_\qc(\spec(k))$ is $\D(k)$, the derived category of $k$-vector spaces. It is a trivial observation that any $k$-vector space is isomorphic to a direct sum of copies of $k$ and also it is a direct sumand of a product of copies of $k$. Furthermore a complex $T\in \D(k)$ is quasi-isomorphic to the complex $\bigoplus_{i\in \ZZ} H^iT[-i]$ with zero differentials. Note also that $T$ is quasi-isomorphic to the coproduct and to the product, 
\[
\bigoplus_{i\in \ZZ} H^iT[-i]=\prod_{i\in \ZZ} H^iT[-i],
\]
of the family of complexes $\big\{\,H^iT[-i]\,;\,i\in \ZZ\,\big\}$.
Therefore the localizing and the colocalizing subcategories of $\D(k)$ are the trivial ones.
\end{cosa}

\begin{cosa}
We will associate localizing and colocalizing subcategories of $\D_\qc(X)$ to subsets of points of  $X$, the underlying topological space of the scheme. To this end,  let
$j_x\colon \Spec(k(x))\to X$
be  the canonical map associated to the residual field $k(x)$ of a given  point $x\in X$,   and set
\[
\CK(x):=\R{j_x}_*k(x)={j_x}_*k(x).
\]
Note that $\CK(x)\in \D_\qc(X)$ is a complex concentrated in degree zero.

For a subset of points $W \subset X$, let  $\SS_W := \{\CK(x) \,/\, x \in W\}$ and let us abbreviate as follows the notation for the  smallest localizing and the smallest colocalizing subcategories of $\D_{\qc}(X)$ determined by the set of complexes $\SS_W$:
\[
\Locset{W} := \Loc(\,{\SS_W}\,), \quad \Colset{W} := \Col(\,{\SS_W}\,)
\]
\end{cosa}

Let us start studying the properties of the localizing and colocalizing subcategories of $\D_{\qc}(X)$ determined by  points. 

\begin{lema}\label{Homvanishes}
The points $x, y\in X$ are different if and only if $\CK(x)\otimes^{\LL}_{X}\CK(y)=0$,  if and only if $\,\R\Hom_{X}(\CK(x),\CK(y))=0$.
\end{lema}

\begin{proof} 
First, $\CK(x)$ is a direct summand of  $\CK(x) \otimes^{\LL}_{X} \CK(x)$, so $\CK(x) \otimes^{\LL}_{X} \CK(x) \neq 0$; and trivially, $\R\Hom_{X}(\CK(x), \CK(x)) \neq 0$.

For $x\neq y$, there exists an affine open subset $U \subset X$ such that
it only contains one of the points, for instance let us assume that
 $x \in U$ and $y \notin U$. Denote by $u : U
\inc X$ the  inclusion. The unit adjunction map $\CK(x)\to \R u_*u^* \CK(x)$ is an isomorphism. Now, using the projection isomorphism
\begin{align*}
      \CK(x) \otimes^{\LL}_{X} \CK(y)
          & \cong \R{}u_*u^*\CK(x) \otimes^{\LL}_{X} \CK(y)    \cong \R{}u_* (u^*\CK(x)
                                      \otimes^{\LL}_U u^* \CK(y)),
\end{align*}
we conclude that $\CK(x) \otimes^{\LL}_{X} \CK(y)=0$ because $u^*\CK(y) = 0$.

Given $\CN \in \D_\qc(X)$, let $\alpha$ be  the natural map
\[
\R\Hom_{X}(\CN, \CK(y)) \overset{\alpha}{\lto}
\R\Hom_{X}(\CN \otimes^{\LL}_{X} \CK(y),
 \CK(y) \otimes^{\LL}_{X} \CK(y)),
\]
and let us consider $\beta$ the map
\[
\R\Hom_{X}(\CN \otimes^{\LL}_{X} \CK(y),\CK(y) \otimes^{\LL}_{X} \CK(y)) 
\overset{\beta}{\lto}
\R\Hom_{X}(\CN, \CK(y)),
\]
determined by the canonical morphisms 
\[
\CN\cong\CN\otimes^{\LL}_{X}\CO_X \lto \CN\otimes^{\LL}_{X}\CK(y),\quad 
\CK(y) \otimes^{\LL}_{X} \CK(y) \lto \CK(y).
\]
It is clear that $\beta \circ \alpha = \id$. Taking $\CN= \CK(x)$, if $x\neq y$ then $\CK(x) \otimes^{\LL}_{X} \CK(y) = 0$, therefore $\beta \circ \alpha$ is the zero map, so 
$\R\Hom_{X}(\CK(x), \CK(y)) = 0$.
\end{proof}

\begin{proposition}
\label{ZandLocZ} For any subset $Z\subset X$:
\[
\begin{aligned}
           Z= & \{\, x\in X \,\,|\,\, \CK(x) \in \Loc (Z)\,\}, \\
X\setminus Z= & \{\, x\in X \,\,|\,\, \CK(x) \in {\Loc (Z)}^\perp\,\}.
\end{aligned}
\]
\end{proposition}

\begin{proof}
It is enough to check that $\CK(x) \in {\Loc (Z)}^\perp$, for all $x\in X\setminus Z$.
By Lemma~\ref{Homvanishes},
the fact that $x\in X\setminus Z$ is equivalent to $\R\Hom_{X}(\CK(z), \CK(x))= 0$ for all $z\in Z$, that is, $\R\Hom_{X}(\CM, \CK(x))= 0$ for all $\CM\in \Loc(Z)$, so $\CK(x) \in {\Loc (Z)}^\perp$.
 \qedhere
\end{proof}

\begin{proposition}
\label{YandColY}
For any subset $Y\subset X$:
\[
\begin{aligned}
           Y= & \{\, x\in X \,\,|\,\, \CK(x) \in \Col (Y)\,\}, \\
X\setminus Y= & \{\, x\in X \,\,|\,\, \CK(x) \in {}^\perp{\Col (Y)}\,\}.
\end{aligned}
\]
\end{proposition}

\begin{proof}
Again, let us check the second identity. 
By Lemma~\ref{Homvanishes},
$x\in X\setminus Y$  is equivalent to $\R\Hom_{X}(\CK(x), \CK(y))= 0$ for all $y\in Y$, that is, $\R\Hom_{X}(\CK(x), \CM)= 0$ for all $\CM\in \Col(Y)$, that is $\CK(x) \in {}^\perp{\Col (Y)}$.
\end{proof}

The next result (not used later) makes explicit the condition of orthogonals in terms of inverse image functors associated to the point.

\begin{proposition}
 Let $x \in X$. The following hold:
 
\begin{enumerate}[1.]
 \item $\{\,\CK(x)\,\}^\perp = \{\,\CF \in \D_{\qc}(X) \,\,|\,\, j_x^\times \CF = 0\,\}$.
 \item ${^\perp\{\,\CK(x)\,\}} = \{\,\CF \in \D_{\qc}(X) \,\,|\,\, \LL j_x^* \CF = 0\,\}$.
\end{enumerate}
\end{proposition}

\begin{proof}
Recall from (\ref{inverse-direct}) that the canonical morphism $j_x\colon \Spec(k(x))\to X$ we have the following adjunctions:
 
 \[
\begin{tikzpicture}
      \node (G) {$\D(k(x))$};  
      \node[node distance=3cm, right of = G] (H) {$\D_{\qc}(X)$};
      \draw[<-, bend left=30] (G) to node [above] {{\footnotesize $\LL j_x^*$}} (H);
      \draw[<-, bend right=30] (G) to node [above] {{\footnotesize $j_x^\times$}} (H);
      \draw[->] (G) to node [above] {{\footnotesize ${j_x}_*$}} (H);
\end{tikzpicture}
\]
where we identify the derived category of $k(x)$-vector spaces,  $\D(k(x))$, with the category $\D_\qc(\Spec(k(x)))$.
 
 Let $\CF \in \D_{\qc}(X)$, we have that 
\[
\R\Hom_{X}(\CK(x), \CF) = \R\Hom_{\{x\}}(\widetilde{k(x)}, j_x^\times \CF) 
\] 
and this implies that $\CF$ is right orthogonal to $\CK(x)$ if and only if the complex of $k(x)$-vector spaces $j_x^\times \CF$ is zero. Analogously, by adjunction,
\[
\R\Hom_{X}(\CF, \CK(x)) = \R\Hom_{\{x\}}(\LL j_x^* \CF, \widetilde{k(x)}) 
\]
but a complex of $k(x)$-vector spaces whose dual is 0 is necessarily 0, therefore we conclude that $\CF$ is left orthogonal to $\CK(x)$ if and only if $\LL j_x^* \CF = 0$.
\end{proof}


\section{Point generated schemes.}\label{sec3}

\begin{cosa}\label{vanishing}
Given $\CF, \CG \in \D_\qc(X)$ (or, in general $\CF, \CG \in \D(X)$), the condition \[\rhom_X(\CG, \CF) = 0\] corresponds to the vanishing of the homologies \ie 
\[
\h^n(\rhom_X(\CG, \CF)) = \Hom_{\D(X)}(\CG, \CF[n]) = 0, \quad \text{ for all } n \in \ZZ.
\]
\end{cosa}

\begin{definition}
We say the scheme $X$ is \emph{point generated} if the set $\{\,\CK(x) \,\,|\,\, x\in X\}$ is a family of generators of $\D_\qc(X)$ as a triangulated category in the sense of \cite[Definition 1.7]{Ngd}, that is, given $\CF \in \D_\qc(X)$, $\rhom_X(\CK(x), \CF) = 0$ for all $ x\in X$ implies $\CF = 0$.
\end{definition}

\begin{lemma}
A scheme $X$ is point generated precisely when $\Locset{X} = \D_\qc(X)$.
\end{lemma}

\begin{proof}
By \cite[Theorem 5.7]{AJS1}, given $\SS$  a set of objects in $\D(X)$ and $\SL$ the smallest localizing subcategory of $\D(X)$ containing $\SS$, the inclusion functor 
$\SL \inc \D(X)$ has a right adjoint $\gamma$. If further $\SS \subset \D_\qc(X)$, then $\gamma \circ \iota$ is the right adjoint to the inclusion $\SL \inc \D_\qc(X)$ where $\iota \colon  \D_\qc(X)\inc \D(X)$ denotes the canonical inclusion functor. As a consequence, we are in the situation of \cite[Proposition 1.6]{AJS1} so $\Locset{X}^\perp = 0$ is equivalent to $\Locset{X} = \D_\qc(X)$.
 \end{proof}

\begin{cosa} By the results in \cite{Nct}, it follows that any noetherian affine scheme is point generated.
In \cite{Oddball}, Neeman built a non noetherian ring $R$ of dimension zero  such that $X=\Spec(R)$ is \emph{not point generated}. In \lc Neeman constructs a very large set of distinct localizing subcategories (Bousfield classes) in $\D_\qc(X)\cong\D(R)$, being $X=\Spec(R)$ a one-point scheme, see section \ref{sec6}. 
\end{cosa}

We will prove next in Theorem~\ref{LocK(x)generan} that any noetherian scheme is point generated.

 \begin{lemma}\label{critpunt}
Let $X$ be a noetherian scheme and let $x\in X$ be a closed point.  For any $\CG\in \D_\qc(X)$, $\R \varGamma_{\{x\}}\CG\in \Locset{\{x\}}$.
 \end{lemma}
 
 \begin{proof}
 (\emph{Cf.}  \cite[\S2]{Nct} for the affine case.) For a closed point $x\in X$, the morphism $j_x\colon \Spec(k(x))\to X$  is a closed embedding. 
  Let $\CQ\subset \CO_X$ be the quasicoherent sheaf of ideals such that ${j_x}_*k(x)=\CO_X/\CQ$. The scheme $X$ is noetherian, so the functor \emph{sections with support} in $\{x\}$, $\varGamma_{\!\{x\}}\colon \SA_\qc(X)\to \SA_\qc(X)$, is computed as
  \[
  \varGamma_{\!\{x\}}(-) =\dirlim{n > 0}\shom_{X}(\CO_X/\CQ^n, -).
  \]
  Therefore for each $\CG\in \D_\qc(X)$
  \[
  \R\varGamma_{\!\{x\}}\CG= \hocolim{n > 0} \HOMqc_{X}(\CO_X/\CQ^n, \CG).
  \]
where we are using B\"okstedt-Neeman notion of countable homotopy colimits \cite[Definition 2.1]{BN}.
Note that $\R {j_x}_*j_x^\times \CG=\HOMqc_{X}(\CO/\CQ, \CG)$ therefore it follows that $\HOMqc_{X}(\CO/\CQ, \CG)\in \Locset{\{x\}}$, also $\HOMqc_{X}(\CQ^n / \CQ^{n+1}, \CG)$ belongs to $ \Locset{\{x\}}$ because the natural map
\[
\R {j_x}_*j_x^\times \HOMqc_{X}(\CQ^n / \CQ^{n+1}, \CG)\cong \HOMqc_{X}(\CQ^n / \CQ^{n+1}, \CG)
\]
is an isomorphism
for any $n\geq 1$. Starting with $n=1$, the existence of the distinguished triangles
\[\small
\HOMqc_{X}(\CO_X/\CQ^{n}, \CG) \lto
\HOMqc_{X}(\CO_X/\CQ^{n+1}, \CG) \lto 
\HOMqc_{X}(\CQ^n / \CQ^{n+1}, \CG) \overset{+}{\lto},
\]%
yields by induction 
that 
$
\HOMqc_{X}(\CO_X/\CQ^{n+1}, \CG)\in \Locset{\{x\}},
$
for all $n\geq 1$.
As a consequence  $ \R\varGamma_{\!\{x\}}\CG$ belongs to  $\Locset{\{x\}}$.
\qedhere
\end{proof}

\begin{theorem}
\label{LocK(x)generan}
A noetherian scheme $X$  is point generated, \emph{i.e.} $\Locset{X}=\,\D_\qc(X)$.\end{theorem}

\begin{proof} Let  $\CF\in \D_\qc(X)$ be any complex and $\CC$ be the family of
stable for specialization subsets $Z \subset X$  such that 
$\R{}\varGamma_Z \CF \in \Locset{X}$. 
If $\{W_\alpha\}_{\alpha \in I}$ is any chain of elements of
$\CC$, and $\CF \to \CJ$ is a K-injective resolution of $\CF$, by \cite[Theorem 2.2 and
Theorem 3.1]{AJS1},
\[
\R{}\varGamma_{\cup W_\alpha} \CF = \varGamma_{\cup W_\alpha} \CJ =
\dirlim{\alpha \in I} \varGamma_{W_\alpha} \CJ ,
\]
belongs to $\Locset{X}$, because  $\R{}\varGamma_{W_\alpha} \CF = \varGamma_{W_\alpha} \CJ \in
\Locset{X}$, for all $\alpha$. Therefore  $\cup W_\alpha \in \CC$.

So there exists a maximal element $W\in \CC$.
Suppose $X \setminus W \neq \emptyset$. As $X$ is
noetherian the family of closed subsets
\[
\CT = \Big\{\, \,\overline{\{z\}} \,\,\,\,|\,\,\,\, z\in X \text{ and } 
\overline{\{z\}} \cap (X \setminus W) \neq \emptyset \,\, \Big\}
\]
has a minimal element $\overline{\{y\}}\in \CT$. If $x \in \overline{\{y\}} \cap
(X \setminus W)$, then $\overline{\{x\}} \in \CT$, but
$\overline{\{y\}}$ is minimal, so $x = y $ and therefore $W \cup \overline{\{y\}} = W
\cup \{y\}$. Let  
$i_y : \Spec(\CO_{X,y}) \to X$ be  the canonical map and let us denote by $X_y\subset X$  the set of points in $\Spec(\CO_{X,y})$. The image through $\R{}\varGamma_{W \cup \{y\}}$ of the  distinguished triangle
\[
\R{}\varGamma_{X \setminus X_y} \CF \lto \CF 
\lto \R{}i_{y\,*}i_y^*\CF \overset{+}{\lto}
\]
 is the triangle
\[ 
\R{}\varGamma_W \CF \lto \R{}\varGamma_{W \cup \{y\}} \CF 
\lto \R{}\varGamma_{\overline{\{y\}}}\R{}i_{y\,*}i_y^*\CF \overset{+}{\lto},
\]
where the third point $\R{}\varGamma_{\overline{\{y\}}}\R{}i_{y\,*}i_y^*\CF\cong \R{}i_{y\,*}\R{}\varGamma_{{\{y\}}}i_y^*\CF \in
\Locset{\{y\}} \subset \Locset{X}$ by Proposition~\ref{critpunt}. As a consequence  $W \cup \{y\}\in \CC$, contradicting the maximality of $W$ in $\CC$.
Necessarily, $W = X$, therefore  $\CF \cong \R{}\varGamma_X \CF \in \Locset{X}$. \qedhere
\end{proof}

\begin{remark*}
This theorem is the usual scheme case of \cite[Proposition 4.4.]{AJS3}. The proof presented here is more direct and helps to make the paper self-contained. 
\end{remark*}

\begin{cosa} \label{absflat} \textbf{Non-noetherian point generated schemes}.
There are examples of non-noetherian point generated schemes, even affine schemes. These examples are due to Stevenson \cite{St14}. We will discuss them
in the last section. 
\end{cosa}

\section{Localizing subcategories determined by points.}\label{sec4}

Throughout this section we will assume that $X$ is a concentrated scheme.

\begin{lema}\label{ayudatensor}
Given a point $x\in X$,  for each $\CM\in\D_{\qc}(X)$,  the complex  $\CK(x)\otimes^{\LL}_{X}\CM$ (equivalentely  $\CM\otimes^{\LL}_{X}\CK(x)$)  belongs to both $\Colset{\{x\}}$ and $\Locset{\{x\}}$.
\end{lema}

\begin{proof}
Let   $j_x\colon \Spec(k(x))\to X$  be the canonical map. By the projection formula
\[
    \CK(x) \otimes^{\LL}_{X} \CM
        =     j_{x*} k(x) \otimes^{\LL}_X \CM 
        \cong j_{x*}\LL j^*_x \CM.
\]
The functor ${j_x}_*\colon \D(k(x))\to \D_\qc(X)$ preserves products and coproducts, being $j_x$ a concentrated morphism. So the complex $\CK(x) \otimes^{\LL}_{X} \CG$ belongs to $\Colset{\{x\}}$ and $\Locset{\{x\}}$, because it is the  image through ${j_x}_*$  of  a complex  of $k(x)$-vector spaces. 
\end{proof}

\begin{proposition}
\label{LZtensorideal}
Let $Z\subset X$ be any subset. If $\CN\in \Locset{Z}$, then $\CM \otimes^{\LL}_X \CN \in \Locset{Z}$, for all $\CM\in \D_{\qc}(X)$.
\end{proposition}
\begin{proof} 
For each $\CM\in \D_{\qc}(X)$, the full subcategory $\SL_{\CM,Z} \subset \D_{\qc}(X)$ defined by
\[
\SL_{\CM,Z}:=\big\{\,\,\CN\in \Locset{Z} \,\,\,\,\big{|}\,\,\,\, \CM\otimes^{\LL}_{X}\CN\in \Locset{Z}\,\,\big\}
\]
is a localizing subcategory in  $\D_{\qc}(X)$ beacuse $\CM\otimes^{\LL}_{X}-\colon \D_\qc(X)\to  \D_\qc(X)$ is a $\Delta$-functor that preserves coproducts and $\Locset{Z}$ is localizing. For each $x\in Z$, by Lemma \ref{ayudatensor}, $\CM\otimes^{\LL}_{X}\CK(x)\in \Locset{\{x\}}\subset \Locset{Z}$. Thus $\SL_{\CM,Z}$ is a localizing subcategory of $\Locset{Z}$ that contains all the complexes $\CK(x)$ with $x\in Z$, so we conclude that $\SL_{\CM,Z}=\Locset{Z}$.
\end{proof}

\begin{cosa} \textbf{Tensor ideals}.
We say that a localizing subcategory $\SL\subset \D_{\qc}(X)$ is $\otimes$-\emph{ideal} if $\CM\otimes^{\LL}_{X} \CN\in \SL$, for all $\CN\in \SL$ and all $\CM\in \D_{\qc}(X)$.
\end{cosa}

On affine schemes, all localizing subcategories are $\otimes$-ideals:

\begin{proposition}\label{allideals}
If $X=\spec(R)$ is an affine scheme then \emph{any} localizing subcategory $\SL\subset \D_{\qc}(X)$ is $\otimes$-{ideal}
\end{proposition}

\begin{proof} For simplicity, let us identify $\D_{\qc}(X)$ and $\D(R)$.
Note that 
\[
\SL':=\big\{\,\,M\in \D(R) \,\,\,\,\big{|}\,\,\,\, M\otimes^{\LL}_{R}N\in \SL,\,\,\forall N\in \SL \,\,\big\}
\]
is a localizing subcategory of $\D(R)$ that contains $R$, then $\SL'=\D(R)$ and, as a consequence, $\SL$ is $\otimes$-ideal.
\end{proof}

\begin{remark*}
 On a general noetherian scheme $X$ there may exist localizing subcategories that are not $\otimes$-ideals. See \cite[Example, p.~595]{AJS3} for an example with $X = \mathbf{P}^1_k$, $k$ a field. Notice that in \lc $\otimes$-ideals are denominated \emph{rigid}.
 
We may consider also for a subset $\SS \subset  \D_{\qc}(X)$ the smallest localizing $\otimes$-ideal containing $\SS$ denoted $\Locsett{\SS}$. In general $\Locset{\SS} \subset \Locsett{\SS}$ and this inclusion may be strict. For instance, let $\SS = \{\CO_X\}$, in this case $\Locsett{\{\CO_X\}} =  \D_{\qc}(X)$ while $\Locset{\{\CO_X\}}$ may be strictly smaller as the previous example shows.

With this notation, for a localizing subcategory generated by residue fields of a subset $Z \subset X$, Proposition \ref{LZtensorideal} says that $\Locset{Z} = \Locsett{\{\CK(x) \,\,|\,\, x \in Z\}}$. So, we will not need (and will not use) the notation $\Locsett{\SS}$ for a general $\SS$ in the paper.
\end{remark*}

\begin{proposition}
\label{hacenfaltaHomideales}
Let $\SL\subset \D_{\qc}(X)$ be a $\otimes$-ideal localizing subcategory and $x\in X$  a point. Then  $\CK(x)\in \SL$ if, and only if,  there exists $\CN\in \SL$ such that $\CK(x)\otimes^{\LL}_{X}\CN\neq 0$.
\end{proposition}

\begin{proof} For any $x\in X$, $\CK(x)\otimes^{\LL}_{X}\CK(x)\neq 0$. If $\CK(x)\in \SL$, the statement holds taking  $\CN=\CK(x)$.

Conversely, let $\CN\in \SL$ and $x\in X$ be such that $\CK(x)\otimes^{\LL}_{X}\CN\neq 0$. The localizing subcategory $\SL$ is a $\otimes$-ideal, therefore $\CK(x)\otimes^{\LL}_{X}\CN$ belongs to $\SL$. From the projection isomorphism
\[
\CK(x)\otimes^{\LL}_{X}\CN =  {j_x}_*k(x)\otimes^{\LL}_{X}\CN \cong {j_x}_*(k(x)\otimes^{\LL}_{k(x)} \LL j_x^*\CN) = {j_x}_* \LL j_x^*\CN,
\]
we conclude that $\CW := \LL j_x^*\CN$ is a non acyclic complex of $k(x)$-vector spaces. Let $i\in \NN$ such that $H^i\CW\neq 0$, then $k(x)[-i]$ is a direct summand of $\CW$. Therefore $\CK(x)[-i] = {j_x}_*k(x)[-i]$ is a direct summand of ${j_x}_*\CW\in \SL$, so $\CK(x)\in \SL$.
\end{proof}

\begin{corollary}\label{ayuda1.1}
Let $\SL \subset \D_{\qc}(X)$ be a $\otimes$-ideal localizing subcategory. Given a point $x\in X$, $\CK(x)\notin \SL$ if, and only if, $\CK(x)\otimes^{\LL}_{X}\CN= 0$ for all $\CN\in \SL$.
\end{corollary}
\begin{proof}
This is just a restatement of Proposition~\ref{hacenfaltaHomideales}.
\end{proof}

\begin{proposition}
\label{Ltensorideal}
Let us assume that $X$ is a point generated concentrated scheme.  If $\SL\subset \D_{\qc}(X)$ is $\otimes$-ideal localizing subcategory   then $\SL=\Locset{Z}$, with  $Z:=\{\,x\in X \,\,|\,\, \CK(x)\in \SL\, \}$.
\end{proposition}
\begin{proof} 
Let $\SD\subset \D_{\qc}(X)$ be the  localizing subcategory  defined by
\[
\SD:=\big\{\,\,\CF\in \D_{\qc}(X)\,\,\,\,\big{|}\,\,\,\,\CF\otimes^{\LL}_{X}\CN\in \Locset{Z}, \,\,\forall \CN\in \SL\,\,\big\}.
\]
This localizing category is a $\otimes$-ideal. To this end, given $\CM\in \D_{\qc}(X)$ and $\CF\in \SD$, let us show that $(\CM\otimes^{\LL}_{X}\CF)\otimes^{\LL}_{X}\CN$ belongs to $\Locset{Z}$, for all $\CN\in \SL$. Indeed, the complex $\CF\otimes^{\LL}_{X}\CN\in \Locset{Z}$, by hypothesis. Then $\CM\otimes^{\LL}_{X}(\CF\otimes^{\LL}_{X}\CN)$ belongs to $\Locset{Z}$ because it is a $\otimes$-ideal by Proposition~\ref{LZtensorideal}, that is, $(\CM\otimes^{\LL}_{X}\CF)\otimes^{\LL}_{X}\CN \in \Locset{Z}$.

The localizing category $\SD$ contains all  the objects $\CK(x)$, with $x\in X$. If $x\in Z$, $\CK(x)\otimes^{\LL}_{X}\CN\in \Locset{Z}$, for any $\CN\in \D_{\qc}(X)$ in particular for those $\CN\in \SL$, then $\CK(x)\in \SD$. Let $x \in X\setminus Z$, that is, such that $\CK(x)\notin \SL$, then by Corollary~\ref{ayuda1.1}, for all $\CN\in \SL$,   $\CK(x)\otimes^{\LL}_{X}\CN=0$ belongs to $\Locset{Z}$. So 
$\SD=\Locset{X}$.

The scheme $X$ is point generated, therefore  $\SD=\D_{\qc}(X)$. In particular $\CO_X\in \SD$, then for all $\CN\in \SL$, $\CN=\CO_X\otimes^{\LL}_X \CN\in \Locset{Z}$. As consequence $\SL=\Locset{Z}.$
\qedhere
\end{proof}

\begin{theorem}\label{ThclasB} Let $X$ be a point generated concentrated scheme. There is a bijection
\[
\left\{    
    \begin{gathered}    
        \text{ Localizing }\\
        \otimes\text{-ideals of\/ } \D_{\qc}(X)
    \end{gathered}
\right\}
\xtofrom[\Loc]{\SZ}
\left\{  
    \begin{gathered}    
        \text{ Subsets of }\\
        X
    \end{gathered}
\right\}
\]
where the map $\Loc$ takes a subset $Z \subset X$ to the localizing subcategory $\Locset{Z}$ generated by the set $\{\CK(x) \,\,|\,\, x \in Z\}$ and $\SZ$ takes a localizing $\otimes$-ideal $\SL$ to the set of points $\SZ(\SL) := Z_\SL=\{\,x\in X \,\,|\,\, \CK(x)\in \SL\,\}$.
\end{theorem}

\begin{proof}
By Proposition~\ref{ZandLocZ}, $Z=\big\{\,x\in X \,\,|\,\, \CK(x)\in \Locset{Z}\,  \big\}$, thus $Z = Z_{\Locset{Z}}$. Conversely, by Proposition~\ref{Ltensorideal}, $\SL=\Locset{Z_\SL}$ for any $\otimes$-ideal localizing subcategory $\SL\subset \D_\qc(X)$.
\qedhere
\end{proof}

\begin{corollary} \label{nonsei} Every $\otimes$-ideal  localizing subcategory $\SL$ of\/
$\D_\qc(X)$ has associated a localization functor.
\end{corollary}

\begin{proof} Theorem \ref{ThclasB} says that a $\otimes$-ideal localizing
subcategory $\SL \subset \D_\qc(X)$ is the smallest localizing subcategory of $\D_\qc(X)$ that contains the set $\{\CK(x) \,\,|\,\, x \in Z_{\SL}\}$. By applying \cite[Theorem 5.7]{AJS1} to the localizing subcategory $\SL$ we conclude that it has an associated localization functor.
\end{proof} 

\section{Colocalizing subcategories determined by points.}\label{sec5}

Let $X$ be a concentrated scheme. For a colocalizing subcategory  $\SC\subset \D_{\qc}(X)$, let $Y_\SC\subset X$ be the set of points
\[
 Y_\SC  := \{\, x\in X\ |\ \CK(x) \in \SC \,\}
\]
We will characterize those colocalizing subcategories of $\D_{\qc}(X)$ determined by the assigned set of points.

\begin{lema}\label{ayuda12.2}
Given a point $x\in X$, each of the complexes  
$\HOMqc_X(\CF, \CK(x))$ and $\HOMqc_X(\CK(x),\CF)$ belongs to both $\Colset{\{x\}}$ and $\Locset{\{x\}}$, for any $\CF\in\D_{\qc}(X)$.
\end{lema}

\begin{proof}
The  canonical map $j_x\colon \Spec(k(x))\lto X$ is a concentrated morphism.
Let us consider the  natural isomorphisms
\[
\begin{aligned}
\HOMqc_X(\CF, \CK(x)) &=  \HOMqc_X\big(\CF, {j_x}_*k(x)\big)\\
                   &\cong {j_x}_*\HOMqc_{\{x\}}\big(\LL j_x^{*}\CF, \widetilde{k(x)}\big),\\
\HOMqc_X(\CK(x), \CF) &=  \HOMqc_X\big({j_x}_*k(x), \CF \big)\\ 
                   &\cong {j_x}_*\HOMqc_{\{x\}}\big(\widetilde{k(x)}, j_x^\times \CF\big).
\end{aligned}
\]
The $\Delta$-functor ${j_x}_*\colon \D(k(x))\to \D_\qc(X)$ preserves products and coproducts. So both complexes $\HOMqc_X(\CF, \CK(x))$ and $\HOMqc_X(\CK(x), \CF)$ belong to $\Colset{\{x\}}$ and $\Locset{\{x\}}$, because they are in the essential image of ${j_x}_*$. 
\end{proof}

\begin{proposition}
Let $Y\subset X$ be a set of points. Then $\HOMqc_X(\CF,\CG)\in \Colset{Y}$ for any $\,\CG\in\Colset{Y}$ and $\CF\in \D_\qc(X)$.
\end{proposition}
\begin{proof}
By Lemma~\ref{ayuda12.2}, for any $x\in Y$,  $\HOMqc_X(\CF, \CK(x))\in \Colset{\{x\}}\subset \Colset{Y}$.
Let
$\SC$ be the full subcategory of $\Colset{Y}$ whose objects are those $\CG\in \Colset{Y}$ such that $\HOMqc_X(\CF,\CG)\in \Colset{Y}$. The category $\SC$ is a colocalizing  subcategory of $\Colset{Y}$  that contains the complexes $\CK(x)$, with $x\in Y$. Therefore $\SC=\Colset{Y}$.\qedhere
\end{proof}

\begin{cosa}\textbf{Hom coideals}.
Given  $\SC\subset \D_{\qc}(X)$ a colocalizing subcategory, we say that $\SC$ is $\CH$-\emph{coideal} if $\HOMqc_X(\CF, \CG ) \in \SC$ for all $\CG \in \SC$ and all $\CF \in \D_{\qc}(X)$.
\end{cosa}

On affine schemes, all colocalizing subcategories are $\CH$-coideals:

\begin{proposition}\label{allcoideals}
If $X=\spec(R)$ is an affine scheme then any colocalizing subcategory $\SC\subset \D_{\qc}(X)$ is an $\CH$-coideal.
\end{proposition}
\begin{proof} Again, for simplicity let us identify $\D_{\qc}(X)$ and $\D(R)$.
Note that 
\[
\SL:=\big\{\,\,M\in \D(R) \,\,\,\,\big{|}\,\,\,\, \R\Hom_R(M, N)\in \SC,\,\,\forall N\in \SC \,\,\big\}
\]
is a localizing subcategory of $\D(R)$ that contains $R$, then $\SL=\D(R)$.\qedhere
\end{proof}

\begin{proposition}\label{ayuda2}
Let $\SC	\subset \D_{\qc}(X)$ be an $\CH$-coideal colocalizing subcategory and $x\in X$ a point.  Then $\CK(x)\in \SC$ if, and only if, there exists a complex $\CG\in \SC$ such that $\HOMqc_X(\CK(x),\CG)\neq 0$.
\end{proposition}
\begin{proof} For each point $x\in X$, $\HOMqc_X(\CK(x), \CK(x))\neq 0$, therefore if $\CK(x)\in \SC$, for $\CG=\CK(x)\in \SC$, $\HOMqc_X(\CK(x),\CK(x))\neq 0$.

Conversely, let $\CG\in \SC$ be such that $\HOMqc_X(\CK(x),\CG)\neq 0$. The colocalizing subcategory $\SC$ is $\CH$-coideal, then $\HOMqc_X(\CK(x),\CG)$ belongs to $\SC$, from the existence of the the isomorphism
\[
\HOMqc_X(\CK(x), \CF) \cong  
{j_x}_* \HOMqc_{\{x\}}\big(\widetilde{k(x)}, j_x^\times \CF\big),
\]
we conclude that $\CV:=\HOMqc_{\{x\}}\big(\widetilde{k(x)}, j_x^\times \CF\big)$ is a non acyclic complex of $k(x)$-vector spaces. Let $i\in \NN$ such that $H^i\CV\neq 0$, then $k(x)[-i]$ is a direct summand of $\CV$. Therefore $\CK(x)[-i] = {j_x}_*k(x)[-i]$ is a direct summand of ${j_x}_*\CV\in \SC$, and as a consequence $\CK(x)\in \SC$.
\end{proof}

\begin{corollary}\label{ayuda2.1}
Let $\SC \subset \D_{\qc}(X)$ be an $\CH$-coideal colocalizing subcategory. Given a point $x\in X$,  $\CK(x)\notin \SC$ if, and only if, $\HOMqc_X(\CK(x),\CG)= 0$ for all $\CG\in \SC$.
\end{corollary}
\begin{proof}
The statement is equivalent to Proposition~\ref{ayuda2}.
\end{proof}

\begin{corollary}
\label{proposicion2}
Let $\SC	\subset \D_{\qc}(X)$ be an $\CH$-coideal colocalizing subcategory. The subsets
$X\setminus Y_\SC=\big\{\,\,x\in X \,\,|\,\,\CK(x)\in {}^\perp\SC\,\,\big\}$.
\end{corollary}
\begin{proof}
It is a restatement of Corollary~\ref{ayuda2.1}. \qedhere
\end{proof}

\begin{proposition}
\label{CHcoideal}
Let $X$ be a point generated concentrated scheme. If $\SC$ is an $\CH$-coideal colocalizing subcategory of\/ $\D_{\qc}(X)$ then $\SC=\Colset{Y_\SC}$.
\end{proposition}

\begin{proof} Denote $Y = Y_\SC$. The full subcategory of $\D_\qc(X)$ defined by
\[
\SD:=\big\{\,\,\CN\in\D_\qc(X)  \,\,\,\,\big{|}\,\,\,\, \HOMqc_X(\CN, \CG)\in \Colset{Y},\, \forall \CG\in \SC\,\,\big\}
\]
is a localizing subcategory.

Let us check now that the localizing subcategory $\SD$ is a $\otimes$-ideal. Given $\CN\in \SD$, by definiton, for any  $\CG\in \SC$, the complex $\HOMqc_X(\CN, \CG)$ belongs to  $\Colset{Y}$. Let us show that, for any $\CF\in \D_\qc(X)$, the complex $\CF\otimes^{\LL}_X \CN$ belongs to $\SD$. Indeed, the colocalizing subcategory $\Colset{Y}$ is $\CH$-coideal, and $\HOMqc_X(\CN, \CG)\in \Colset{Y}$ then $\HOMqc_X(\CF,  \HOMqc_X(\CN, \CG))\in \Colset{Y}$, for all $\CF\in \D_\qc(X)$. As a consequence, $\HOMqc_X(\CF\otimes^{\LL}_X \CN, \CG)\cong\HOMqc_X(\CF,  \HOMqc_X(\CN, \CG))$ belongs to $ \Colset{Y}$, for all $\CG\in \SC$.

Furthermore, it holds that $\CK(x)\in \SD$ for all $x\in X$. Indeed, if $x\in Y$, by Lemma \ref{ayuda12.2}, $\HOMqc_X(\CK(x), \CG)\in \Colset{\{x\}} \subset \Colset{Y}$ for any complex $\CG$ in $\D_\qc(X)$, so in particular $\CK(x)\in \SD$; if $x\notin Y$ then, by Corollary~\ref{proposicion2},  $\CK(x)\in {}^\perp\SC$, in this case, we have that  $\CK(x)\in \SD$ because $\HOMqc_X(\CK(x), \CG)=0\in \Colset{Y}$, for all $\CG\in \SC$.

The scheme $X$ is point generated,  thus $\SD=\D_\qc(X)$. In particular $\CO_X \in \SD$ and it follows that   $\CG=\HOMqc_X(\CO_X, \CG)$ belongs to  $\Colset{Y}$, for all $\CG\in \SC$.
We conclude that $\Colset{Y}=\SC$.
\end{proof}

\begin{theorem}\label{ThclasCB} Let $X$ be a point generated concentrated scheme. There is a bijection
\[
\left\{    
    \begin{gathered}    
        \text{ Colocalizing }\\
        \CH\text{-coideals of\/ } \D_{\qc}(X)
    \end{gathered}
   \right\}
\xtofrom[\Col]{\SY}
\left\{  
\begin{gathered}    
        \text{ Subsets of }\\
        X
    \end{gathered}
   \right\}
\]
where the map $\Col$ takes a subset $Y \subset X$ into the colocalizing subcategory $\Colset{Y}$ generated by the set $\{\CK(x) \,\,|\,\, x \in Y\}$ and $\SY$ assigns to an $\CH$-coideal colocalizing subcategory $\SC$, the set  $\SY(\SC) := Y_\SC=\{\,x\in X \,\,|\,\, \CK(x)\in \SC\,\}$.
\end{theorem}

\begin{proof}
It is clear that $Y=\big\{\,x\in X \,\,|\,\, \CK(x)\in \Colset{Y}\,  \big\}$ by Proposition \ref{YandColY}.
In the converse sense, by Proposition~\ref{CHcoideal}, $\SC=\Colset{Y_\SC}$ for any $\CH$-coideal colocalizing  subcategory $\SC\subset \D_\qc(X)$.
\qedhere
\end{proof}

\begin{proposition} \label{trestrestres} Let $X$ be a point generated concentrated scheme. The correspondence $\SL\mapsto \SL^{\perp}$  defines a bijection  between $\otimes$-ideal localizing subcategories of $\D_\qc(X)$ and $\CH$-coideal colocalizing subcategories, with inverse the correspondence $\SC\mapsto {}^\perp\SC$.
\end{proposition} 

\begin{proof}
By the tensor-Hom adjunction, if $\SL$ is a $\otimes$-ideal localizing subcategory of $\D_\qc(X)$ then $\SL^\perp$ is an $\CH$-coideal colocalizing subcategory and, reciprocally, if $\SC$ is an $\CH$-coideal colocalizing subcategory of $\D_\qc(X)$ then ${}^\perp\SC$ is a $\otimes$-ideal localizing subcategory. 

Using Theorem \ref{ThclasB}, $\SL = \Locset{Z}$ for a certain $Z \subset X$. Moreover, as $\SL$ is a $\otimes$-ideal, $\SL^\perp$ is an $\CH$-coideal and therefore $\SL^\perp = \Colset{Y}$ where $Y = X \setminus Z$ in view of Proposition \ref{ZandLocZ}. If $\SC$ is an $\CH$-coideal colocalizing subcategory then $\SC = \Colset{Y}$ for a certain $Y \subset X$ by Theorem \ref{ThclasCB}, and by Proposition \ref{YandColY} we obtain that ${}^\perp\SC = \Locset{X \setminus Y}$. As a consequence,  ${}^\perp(\SL^\perp) = \SL$ and $({}^\perp\SC)^\perp = \SC$, thus the correspondences are mutually inverse.
\end{proof}

\begin{corollary}\label{corfin}
Let $X$ be a point generated concentrated scheme. For any partition  $X=Z\cup Y$, it holds that $\Locset{Z}^\perp = \Colset{Y}$ and $\Locset{Z}={}^\perp\Colset{Y}$.
\end{corollary}
\begin{proof}
Follows form the proof of the previous Proposition.
\end{proof}

\section{Further remarks on supports and point generation}\label{sec6}

The following remarks will relate the notion of point generated scheme with some other conditions more general than being a noetherian scheme.

\begin{cosa} Let $X$ be a scheme and $\CF \in \D_\qc(X)$. Let us recall the definition of its (small) support
\[
 \supp(\CF) := \{\,x\in X \,\, | \,\, \CK(x)\otimes^{\LL}_X \CF \neq 0\,\}
 \]
 With this in mind, we say that a scheme $X$ satisfies the \emph{support property} if the following condition holds for $\CF \in \D_\qc(X)$
 \[
\left[ \CK(x)\otimes^{\LL}_X \CF=0,\ \forall\ x\in X \right] \limp \CF = 0, 
 \]
 in other words, $\CF = 0$ if and only if $ \supp(\CF) = \varnothing $.
\end{cosa}

\begin{lema}
 The support property for a scheme $X$ is equivalent to ${}^\perp\Colset{X} = 0$.
\end{lema}

\begin{proof}
 Let $x \in X$ and $j \colon \Spec(k(x))\to X$ the canonical map. For any $\CF \in \D_\qc(X)$,
 \[
\CK(x)\otimes^{\LL}_X \CF=j_*k(x)\otimes^{\LL}_X \CF\cong j_*(k(x)\otimes_{\{x\}}^{\LL}{\LL}j^*\CF)=j_*{\LL}j^*\CF.
\]
Therefore, it follows that:
\begin{align*}
\CK(x)\otimes^{\LL}_X \CF=0 \quad
            \ldimp &\quad {\LL}j^*\CF=0\\
            \ldimp &\quad  \R\Hom_{k(x)}({\LL}j^*\CF, k(x))=0\\
            \ldimp &\quad \R\Hom_{X}(\CF, j_*\widetilde{k(x)})=\R\Hom_{X}(\CF, \CK(x)) = 0
\end{align*}
So $X$ satisfies the support property if and only if ${}^\perp\Colset{X} = 0$.
\end{proof}

\begin{proposition}
 If $X$ is point generated concentrated scheme then $X$ satisfies the support property.
\end{proposition}

\begin{proof}
 By Proposition \ref{trestrestres}, If $X$ is point generated, then ${}^\perp\Colset{X} = 0$.
\end{proof}

\begin{cosa}\textbf{Bousfield classes as localizing subcategories}. Let $X$ be a concentrated scheme. Given $\CE \in \D_\qc(X)$, let $\D_\qc(X)_{\CE}$ be the Bousfield class defined by $\CE$,
\[
\D_\qc(X)_{\CE} =  \{\,\CF \in \D_\qc(X)\,\,|\,\,\CF\otimes^{\LL}_X \CE=0\}.
\]
All Bousfield classes are clearly $\otimes$-ideal localizing subcategories.

If $X$ is point generated, it is a consequence of Corollary \ref{corfin} that 
\[
\D_\qc(X)_{\CE} = \Locset{X \setminus \supp(\CE)}.
\]
\end{cosa}

\begin{cosa}\textbf{Bousfield classes and the support property}.
 We restrict our discussion to the affine case. Let $R$ be a commutative ring and $E \in \D(R)$. The Bousfield class in $\D(R)$ associated to $E$ is the full subcategory
 \[
\D(R)_E = \{\,M\in\D(R) \,\,|\,\, M\otimes^{\LL}_R E=0\}.
\]

In \cite{Oddball}, Neeman gave an example of a non-noetherian ring $R$ with a single prime ideal whose set of Bousfield classes is infinite. Let us describe it briefly. Let $R$ be the ring defined by
\[
R=\dfrac{k[x_2,x_3,\ldots \,]}{\langle x_2^2,x_3^3,\ldots \rangle},
\]
where $k$ is a field and $\{x_2,\,x_3,\ldots\}$ is a countable set of variables.
The ring $R$ is local with maximal ideal $\im=\langle\{x_i\}_{i\geq 2}\rangle$; furthermore, all elements of $\im$ are nilpotent, so $\im$ is the only prime ideal of $R$, that is, $\spec(R)$ has a single point.
However, $\D(R)$ has infinitely many Bousfield classes and hence infinitely many localizing subcategories. In particular, the affine scheme $\spec(R)$ is not generated by points; moreover, as we will see, $\spec(R)$ does not satisfy the support property either.

Let us recall the construction of Bousfield’s classes in \cite{Oddball}. Set $\NN^* := \NN\setminus \{0\}$. Let $P_0$ be the set of functions $f\colon \NN^* \to \NN^*$ such that $f(j)\leq j$ for any $j\in \NN^*$. Consider the following preorder relation on $P_0$
\[
f\preceq g \,\Longleftrightarrow \, \big[\,\,f(j)\leq g(j),\,\forall j\gg 1\,\,\big].
\]
For each $f \in P_0$, denote
\[
E_f=\dfrac{R}{\langle x_2^{{}^{f(2)}},x_3^{{}^{f(3)}},\ldots\rangle}.
\]
The equivalence relation on $P_0$ associated with the preorder is defined, for $f, g \in P_0$, as $f \sim g$ if $f \preceq g$ and $g \preceq f.$ Let $P := P_0 / \sim$ be the quotient set. The set $P$ is ordered by the relation induced from the preorder on $P_0$. Let $\mathcal{B}$ denote the Bousfield classes of $\D(R)$ ordered by the inclusion relation. The mapping
\begin{align*}
   \psi\colon P  & \lto \mathcal{B} \\
             [f] & \lto \D(R)_{E_f}
\end{align*}
is injective (see \cite[Lemma 2.6]{Oddball}) and reverses the order. The set $P$ has at least $2^{\aleph_0}$ elements; therefore, there are infinitely many Bousfield classes in $\D(R)$.

Let $\mathbf{1}\colon \NN^* \to \NN^*$ be the constant map taking the value $1$. The class $[\mathbf{1}]$ is the minimum of the set $P$. Therefore $\D(R)_{E_{\mathbf{1}}}$ is the largest of the Bousfield classes in the set $\Img(\psi)$, in particular $\D(R)_{E_{\mathbf{1}}}\neq 0$. As $E_{\mathbf{1}}=k(\im)$, it follows that 
\begin{equation*}
\exists M\in \D(R), \text{ such that } M\neq 0 \text{ and } M\otimes^{\LL}k(\im)=0.
\end{equation*}
Therefore, the affine scheme $\spec(R)$ does not satisfy the support property.
\end{cosa}

\begin{cosa}
\label{sumup} 
Let us summarize the relationship between the previous conditions on concentrated schemes
\[
\text{noetherian} \limp
\text{point generated} \limp
\text{support property}
\]
We will see in the next paragraph that both implications are strict. Moreover, the weaker condition, the support property, does not always hold even for affine schemes as it is shown by Neeman's previously discussed example. The failure of this condition is related to the complexity of the nilradical.
\end{cosa}

\begin{cosa}\textbf{Absolute flat rings and point generation}.
We stay in the affine case. A ring $R$ is \emph{absolutely flat}  (or \emph{Von Neumann regular}) if $\langle r\rangle=\langle r^2\rangle$, for all $r\in R$. An absolutely flat ring $R$ satisfies the support property because $R_\ip = k(\ip)$, for all $\ip \in \spec(R)$. The affine scheme $\spec(R)$ associated to a non-noetherian, absolutely flat ring $R$ is a zero-dimensional, Hausdorff and totally disconnected scheme and therefore is not a topologically noetherian space.

The inclusion of the subcategory of absolutely flat commutative rings into the category of commutative rings admits a left adjoint, which associates to a commutative ring $S$ an absolutely flat ring defined by
\[
S^{\text{abs}} = \dfrac{S[\,x_s\, \mid\, s\in S\,]}{\big\langle \,\{\,s^2x_s-s, x_s^2s-x_s\, \mid\, s\in S\,\}\,\big\rangle}.
\]
In \cite[Proposition 4.18]{St14} it is shown that the affine scheme $\spec(R)$ where  $R=S^{\text{abs}}$ is point generated if $S$ is such that $\spec(S)$ is a noetherian topological space, thus it provides an example of a non noetherian point generated scheme. This gives a counterexample to the converse of the first implication in \ref{sumup}.

For the counterexample to the second implication let us recall the following notion. A ring $R$ is semi-artinian if, for any proper ideal $\ia\subset S$, the module $R/\ia$ contains a simple $R$-submodule. In \cite[Theorem 4.8]{St14}, Stevenson proves that if $R$ is an absolutely flat and non-semi-artinian ring, then $X=\spec(R)$ is not generated by points. The key fact is the existence of an non-zero injective $R$-module $E$, which is right-orthogonal to all residue fields of $R$, hence $E \notin \Locset{\{k(\ip)\ |\ \ip\in\spec(R)\}}$.

An example of a completely flat, non-semi-artinian ring is as follows. Given an infinite set $\Lambda$ and a family of fields $\{\,k_\lambda \,\,|\,\, \lambda\in \Lambda\}$ denoted by $\Lambda$, the ring
\[
R:=\Pi_{\lambda\in \Lambda}k_\lambda,
\]
is absolutely flat and not semi-artinian. Therefore, $\spec(R)$ is a scheme not generated by points on which the support condition holds. 
\end{cosa}

\begin{remark}
 Our approach differs substantially to the one taken in \cite{bchs}. Essentially, the authors consider certain \emph{idempotent functors} determined by localization and acyclization functors determined by compact objects that play the role of (sheaves associated to) points in the Balmer spectrum of the category. In the context of schemes, these objects would correspond to a localization of a Koszul complex associated to the ideal of $\overline{\{x\}}$ for a point $x$ of the scheme. These complexes are \emph{not} isomorphic to the sheaves $\CK(x)$ used here, in general.

What underlies all this is that in the scheme case we get a system of \emph{points} whose associated derived categories have a remarkable simple structure with respect to localizations and colocalizations. The issue of defining residue fields in general tensor triangulated geometry is delicate as Cameron and Stevenson reference \cite{CS} attests.
\end{remark}




\begin{thebibliography}{ABCD}

\bibitem[AJL]{AJL} Alonso Tarr{\'{\i}}o, L.; Jerem{\'{\i}}as
L{\'o}pez, A.; Lipman, J.: Local homology and cohomology on schemes, 
{\it Ann.\ Scient.\ {\'E}c.\ Norm.\ Sup.\ }{\bf 30}~(1997), 1--39.


\bibitem[AJPV]{asht} Alonso Tarr{\'{\i}}o, L.; Jerem{\'{\i}}as
L{\'o}pez, A.; P\'erez Rodr\'{\i}guez, M.; Vale Gonsalves M. J.:
The derived category of quasi-coherent sheaves and axiomatic stable homotopy.
\textit{Adv. Math.} \textbf{218} (2008) 1224--1252.

\bibitem[AJS1]{AJS1} Alonso Tarr{\'{\i}}o, L.; Jerem{\'{\i}}as
L{\'o}pez, A.; Souto Salorio, M. J.: Localization in categories of
complexes and unbounded resolutions \textit{Canad. J. Math.} \textbf{52}
(2000), no. 2,  225--247.


\bibitem[AJS2]{AJS3} Alonso Tarr{\'{\i}}o, L.; Jerem{\'{\i}}as
L{\'o}pez, A.; Souto Salorio, M. J.: Bousfield localization on formal schemes. \textit{J. Algebra} \textbf{278} (2004), no. 2, 585--610.

\bibitem[BCHS]{bchs} Barthel, T.; Castellana, N.; Heard, D.; Sanders, B.:
Cosupport in tensor triangular geometry, \texttt{arXiv:2303.13480}, available at \url{https://arxiv.org/abs/2303.13480}.


\bibitem[BIK]{bik} Benson, D. J.; Iyengar, S. B.; Krause, H.: 
Colocalizing subcategories and cosupport. \textit{J. Reine Angew. Math.}, \textbf{673} (2012) 161--207.

\bibitem[BN]{BN} B\"okstedt, M.; Neeman, A.: Homotopy limits in
triangulated categories, {\it Compositio Math.\ }{\bf 86}~(1993),
209--234.

\bibitem[BvdB]{BB} Bondal, A.; van den Bergh, M.: Generators and representability of functors in commutative and noncommutative geometry,  {\it Mosc. Math. J.} {\bf 3} (2003), no. 1, 1–36, 258.

\bibitem[CS]{CS} Cameron, J.; Stevenson, G.:
Homological residue fields as comodules over coalgebras. in \textit{Triangulated categories in representation theory and beyond -- the Abel Symposium 2022}, 255--270. \textit{Abel Symp.}, \textbf{17}. Springer, Cham, 2024.

\bibitem[GD]{GD} Grothendieck, A.; Dieudonn\'{e}, J. A.:
{\it El\'{e}ments de   G\'{e}om\'{e}trie Alg\'{e}brique I}. Grundlehren
Math. Wissenschaften  {\bf 166}. Sprin\-ger-Verlag, Berlin-New York, 1971.


%
%
%
%
%
%


\bibitem[I]{erg} Illusie, L.: Existence de r\'esolutions globales, in {\it
Th\'eorie des Intersections et Th\'eor\`eme de Riemann-Roch (SGA~6),}
Lecture Notes in Math., no.\,{\bf 225}, Springer-Verlag, New York, 1971,
160--221.




%
%

\bibitem[Li]{yellow} Lipman, J.: Notes on derived categories and Grothendieck Duality. \textit{Foundations of Grothendieck duality for diagrams of schemes,} 1--259, Lecture Notes in Math., {\bf 1960}, Springer-Verlag, Berlin-New York, 2009.

%
%
%
%
\bibitem[N1]{Nct} Neeman, A.: The chromatic tower for $D(R)$. {\it
Topology}\/ {\bf 31} (1992), no.~3, 519--532.
%


\bibitem[N2]{Ngd} Neeman, A.: The Grothendieck duality theorem via
Bousfield's techniques and Brown representability. {\it J. Amer. Math. Soc.}
{\bf 9} (1996), no. 1, 205--236.


\bibitem[N3]{Oddball} Neeman, A.:  Oddball Bousfield classes. \textit{Topology} \textbf{39} (2000), no. 5, pp. 931--935.

\bibitem[N4]{Ncs} Neeman, A.: Colocalizing subcategories of $\mathbf{D}(R)$. \textit{J. Reine Angew. Math}. \textbf{653} (2011), pp. 221--243. 
	

%
%

\bibitem[S1]{St13} Stevenson, G.:
Support theory via actions of tensor triangulated categories.
\textit{J. Reine Angew. Math.} \textbf{681} (2013), 219--254.

\bibitem[S2]{St14} Stevenson, G.:
Derived categories of absolutely flat rings.
\textit{Homology Homotopy Appl.} \textbf{16} (2014), no. 2, 45--64.

\bibitem[TT]{tt} Thomason, R. W.; Trobaugh, T.: Higher algebraic $K$-theory of schemes and of derived categories. \textit{The Grothendieck Festschrift, Vol. III}, 247--435, Progr. Math., {\bf 88}, Birkh\"auser Boston, Boston, MA, 1990.

\bibitem[V]{v} Verasdanis, C.: Costratification and actions of tensor-triangulated categories, \texttt{arXiv:2211.04139}, available at \url{https://arxiv.org/abs/2211.04139}.




\end{thebibliography}
\end{document}